\documentclass{article}
\usepackage[utf8]{inputenc}
\usepackage{amsmath, amsthm, amssymb, enumerate, multicol, chngcntr}

\theoremstyle{plain}
\newtheorem{Thm}{Theorem}[section]
\newtheorem{Prop}[Thm]{Proposition}
\newtheorem{Cor}[Thm]{Corollary}
\newtheorem{Lem}[Thm]{Lemma}

\theoremstyle{definition}
\newtheorem{Def}[Thm]{Definition}
\newtheorem{exmp}[Thm]{Example}

\counterwithin{equation}{section}

\newcommand{\FG}{F[G]}
\newcommand{\FH}{F[H]}
\newcommand{\FK}{F[K]}
\renewcommand{\S}{\mathfrak{S}}
\newcommand{\T}{\mathfrak{T}}
\newcommand{\D}{\mathcal{D}}
\newcommand{\E}{\mathcal{E}}
\renewcommand{\H}{\mathcal{H}}
\newcommand{\Z}{\mathcal{Z}}
\renewcommand{\P}{\mathcal{P}}

\DeclareMathOperator{\supp}{supp}
\DeclareMathOperator{\stab}{Stab}
\DeclareMathOperator*{\Span}{Span}
\DeclareMathOperator*{\Aut}{Aut}

\title{On Schur Rings Over Infinite Groups}
\author{Nicholas Bastian, Jaden Brewer, Andrew Misseldine}

\date{\today}

\begin{document}

\maketitle

\begin{abstract}
Schur rings are a type of subring of the group ring that is spanned by a partition of the group that meets certain conditions. Past literature has exclusively focused on the finite group case. This paper extends many classic results about Schur rings to the infinite groups, including Leung-Man's classification of Schur rings over finite cyclic groups. Additionally, this classification will be extended to all torsion-free locally cyclic groups. \end{abstract}

\textbf{Keywords}:
Schur Ring, cyclic group, locally cyclic group, torsion-free abelian group, Laurent polynomial ring, association scheme\\

\textbf{MSC Classification}:
20c07, 
16s34  
20e25 
20k20 
20c05, 
05e30 

\section{Introduction}
Schur rings (or S-rings) are certain subrings of the group ring afforded by a partition of the group which satisfy certain conditions (see Definition \ref{def:sring}). They were originally developed by Schur \cite{Schur33} and Wielandt \cite{Wielandt49} in the first half of the 20th century with the original purpose of studying permutation groups as an alternative to character theory and have since been used in many other applications. Of particular importance is Schur rings' connection to algebraic combinatorics first mentioned by Klin and R. P{o}schel \cite{KlinPoschel}, after which Schur rings have been used in the study of association schemes, strongly regular graphs, and difference sets. They were especially studied in this regard in the 1980s and 1990s for the case of finite cyclic groups, culminating in the complete classification of Schur rings over finite cyclic groups by Leung and Man \cite{LeungII,LeungI}. Recent applications of Schur rings include super character theories of groups noted by Hendrickson \cite{Hendrickson} and Hadamard matrices noted by Lopez \cite{Lopez18, Lopez18a}. The representation theory of Schur rings was originally developed by Tamaschke \cite{Tamaschke69,Tamaschke70} as a purely algebraic object. Wielandt's monograph on permutation groups \cite{Wielandt64}, as well as Scott \cite{Scott}, Muzychuk and Ponomarenko \cite{Muzychuk09}, Hienze \cite{Heinze}, and Misseldine \cite{MePhD}, provide introductory coverage of Schur rings over finite groups and their applications.

Thus far, the theory and applications of Schur rings has focused exclusively on Schur rings over finite groups. Two notable exceptions are Tamaschke \cite{Tamaschke69} and Haviv and Levy \cite{Haviv18}, which introduce generalizations of Schur rings, called S-semigroups and Schur dioids, respectively. In both cases, although the new structure generalizes the notion of Schur rings and allow the use of infinite groups,\footnote{In fact, Tamaschke specifically mentions that S-semigroups as a ``structure has wider applications. It is not bound to finite groups.'' \cite{Tamaschke69}. With the proper definition of Schur ring, we refute Tamaschke's claim that Schur rings are incompatible with infinite groups.} they are not subrings of the group algebra, in general. This paper provides a definition of Schur rings which includes infinite groups and coincedes with Wieldant's original definition in the case of finite groups. From this extension to infinite groups it is hoped that applications of Schur rings can likewise be extended to infinite permutation groups, infinite graphs, super characters of infinite groups, etc. 

Section \ref{sec:schur} will extend Wielandt's notion of Schur modules and rings to infinite groups, including many important properties due to Wielandt and others. Section \ref{sec:cyclic} will provide a complete classification of Schur rings over the infinite cyclic group. The group ring over the infinite cyclic group can, of course, be identified with the ring of Laurent polynomials. Leung and Man provided a classification of Schur rings over finite cyclic groups by showing that all Schur rings fall into one of four families: partitions induced by orbits of automorphism subgroups (orbit Schur rings), partitions induced by a direct product factorization (dot product), partitions induced by cosets (wedge products), and the trivial Schur ring. For the infinite cyclic group, it will be shown that there are exactly two Schur rings over this group, namely the group ring itself and the symmetric ring (see Theorem \ref{thm:int}). As both of these Schur rings fall under the orbit Schur ring family, this result extends the classification of Schur rings to all cyclic groups. This will be primarily accomplished using an operation on Schur rings which we call freshman exponentiation, an operation known to Wielandt. Section \ref{sec:rational} generalizes Theorem \ref{thm:int} to all torsion-free, locally cyclic groups (see Theorem \ref{Thm:TFLC}).

The following notation will be used throughout. Let $G$ denote an arbitrary group 
and $F$ will denote an arbitrary field of characteristic zero. The group ring over $G$ with coefficients in $F$ will be denoted $\FG$. For $\alpha\in \FG$ we say $\alpha=\sum_{g\in G} \alpha_g g$ where $\alpha_g\in F$ and $\alpha_g=0$ for all but finitely many $g$. The group $\Z = \langle z\rangle$ will denote the infinite cyclic group, written multiplicatively, while $\mathbb{Z}$ will denote the usual ring of integers.

\section{Schur Rings}\label{sec:schur}

Let $\alpha \in \FG$. Then define $\alpha^{\ast}=\sum_{g\in G} \alpha_g g^{-1}$, which is an involution on $\FG$. For any finite subset $C\subseteq{G}$ we say that $\overline{C}=\sum_{g\in C} g$ is a \emph{simple quantity}. Additionally, the \emph{Hadamard product} (or the \emph{circle product}) is defined as $\alpha\circ \beta=\sum_{g\in G} \alpha_g\beta_g g$, which is an associative, commutative, bilinear multiplication. In particular, the set $\FG$ equipped with the Hadamard product forms a ring isomorphic to the direct sum of $|G|$-many copies of $F$. Note if $|G|< \infty$, then $\overline{G}$ is the identity of this operation, but if $|G|=\infty$ then $\circ$ has no identity. Two other important facts to note are that the simple quantities are exactly the central idempotent elements with respect to Hadamard multiplication and that $\overline{C}\circ \overline{D}=\overline{C\cap D}$. Let $\alpha^{\circ n} = \alpha\circ\ldots\circ \alpha$ be the $n$-fold Hadamard product of $\alpha$ for $n\in \mathbb{N}$. The \emph{support} of an element $\alpha\in \FG$, denoted $\supp(\alpha)$, is $\{g \mid \alpha_g\neq{0}\} \subseteq G$, which must necessarily be finite. Note that $C=\supp(\overline{C})$ and $C^* = \supp(\overline{C}^*)$.

Let $\S$ be a $F$-subspace closed under $\circ$. Let $\alpha\in \S$ be a simple quantity. We say that $\alpha$ is a \emph{primitive quantity} if $\alpha$ is not the sum of two nonzero simple quantities in $\S$. Alternatively, $\alpha$ is primitive if $\alpha\circ \beta = 0$ whenever $\beta$ is a simple quantity in $\S$ such that $\supp(\alpha)\not\subseteq \supp(\beta)$. Finally, the primitive quantities of $\S$ are exactly the primitive central idempotents when $\S$ is viewed as a subring of $(\FG,+,\circ)$. This implies that $\S$ is the $F$-span of its primitive quantities since each element has a local identity, namely its support. Let $C\in \D(\S)$, where $\D(\S)$ is the partition of the Schur ring $\S$ over $F[\mathcal{Z}]$. This also implies that if $\alpha, \beta\in \S$ where $\alpha$ is a primitive quantity, then $\alpha\circ\beta=r\alpha$ where $r\in F$.

Traditionally, Schur modules (and rings) have only been constructed over finite groups. Other than known connections to other branches of mathematics, such as algebraic combinatorics, this primarily has been due to the restriction that $\overline{C}$ is a well-defined element of $\FG$ if and only if $C$ is finite. Thus, when extending the notion of a Schur module to infinite groups, we will require that the classes $C$ of the associated partition $\P$ be finite, that is, if $C\in \P$ then $|C| < \infty$, called a partition of \emph{finite support}. Necessarily, if $G$ is an infinite group and $\P$ is a partition of finite support, then $|\P| = |G|$. The following formal definition of Schur modules, originally due to Wieldandt \cite{Wielandt49}, has accordingly been modified to allow the group $G$ to be either finite or infinite.

\begin{Def}[Schur Module]\label{def:smodule}
Let $G$ be a group and let $\S$ be an $F$-subspace of $\FG$. We say that $\S$ is a \emph{Schur module} (or an \emph{S-module}) if there exists a partition $\P$ of $G$ of finite support such that $\S=\Span_F\{\overline{C} \mid C\in \mathcal{P}\}$.  
\end{Def}

For a Schur module $\S$, we denote the associated partition $\mathcal{P}$ as $\D(\S)$. We call the sets $C\in \D(\S)$ the $\S$-\emph{classes} or the \emph{primitive sets} of $\S$. We also call $\overline{C}$ a \emph{primitive quantity}\footnote{Due to Proposition \ref{lem:had}, the two notions of primitive quantity agree for Schur modules.} if $C$ is a primitive set.

Let $\S$ be a Schur module over a group $G$ and let $C \subseteq G$. We say $C$ is an \emph{$\S$-set} of $G$ if there exists a subset $\E\subseteq \D(\S)$ (not necessarily finite) such that $C=\bigcup_{E\in \E} E$.  An $\S$-set $G$ is an \emph{$\S$-subgroup} if it is likewise a subgroup of $G$.

Many of the properties of Schur modules over finite groups remain true for the infinite case. The following propositions are due originally to Wielandt \cite{Wielandt64} and most carry over to the infinite case with no significant modification to their proofs.\footnote{For most of these results, Wielandt provided little or no proof. See \cite{MePhD} for these omitted proofs.} 

\begin{Prop}[Wielandt, Proposition 22.1]\label{lem:s-set} 
Let $G$ be a group and let $\S$ be a Schur module over $\FG$. Let $\alpha \in \S$ such that $\alpha=\sum_{g\in G} \alpha_g g$. Then $K(\alpha,c)=\{g\in G \mid \alpha_g=c\}$ is an $\S$-set for each $c\in F$, called the \textit{coefficient complex} associated to $\alpha$ with respect to $c$.
\end{Prop}


\begin{Prop}[Wielandt, Proposition 22.2]\label{lem:sup} 
Let $\S$ be a Schur module over $\FG$. Let $\alpha=\sum_{g\in G} \alpha_g g\in \S$. Then $\supp(\alpha)$ is an $\S$-set.
\end{Prop}


\begin{Prop}[Wielandt, Proposition 22.3]\label{lem:wielandt} 
Let $\S$ be a Schur module over a group $G$ and let $\alpha=\sum_{g\in{G}} \alpha_g g$. Let $f : F \to F$ be a function such that $f(0)=0$. We define $f[\alpha] = \sum _{g\in{G}} f(\alpha_g) g$. Then $f[\alpha]\in \S$. 
\end{Prop}

When $|G|< \infty$, the requirement that $0$ be a fixed point of $f$ may be removed.

\begin{Prop}[Wielandt, Proposition 22.4]\label{lem:had} 
Schur modules are closed under the Hadamard product.
\end{Prop}

Let $f(x) = c_1x + c_2x^2 +\ldots + c_nx^n \in F[x]$. Then for $\alpha \in \S$, let $f^\circ(\alpha) = c_1\alpha + c_2\alpha^{\circ 2} + \ldots + c_n\alpha^{\circ n}$.  If $\alpha\in \S$ for any Schur module, then $f^{\circ}(\alpha)\in \S$ by Proposition \ref{lem:had}. By virtue of the Hadamard product, $f^\circ(\alpha) = \sum_{g\in G} f(\alpha_g)g = f[\alpha]$, as defined in Proposition \ref{lem:wielandt}. In particular, if $\S$ is any $F$-subspace of $\FG$ closed under $\circ$, then $f[\alpha]\in \S$ whenever $\alpha\in \S$.

Thus, Schur modules are subrings of $\FG$ with respect to the Hadamard product. For finite groups, the converse of this statement is likewise true. A proof of this is provided by Muzychuk \cite[Lemma 1.3]{Muzychuk94}, which relies on the semisimplicity of $(\FG, +,\circ)$. As this ring fails to be semisimple when $|G|=\infty$, a modified proof is included below.

\begin{Lem}\label{lem:Smod}
If $\S$ is an $F$-subspace of $\FG$ closed under $\circ$, then $\overline{K(\alpha,c)}\in \S$ for all $\alpha\in \S$ and $c\in F\setminus \{0\}$.
\end{Lem}

\begin{proof}
We proceed by induction on $n$ the number of non-zero distinct coefficients of $\alpha$. If $n=1$, then $\alpha = c\overline{K}$ for some coefficient $c\in F$ and finite subset $K \subseteq G$. Then $\overline{K} = \frac{1}{c}\alpha\in \S$. 

If $\overline{K(\alpha,c)}\in \S$ whenever $\alpha$ has fewer than $n$ nonzero coefficients, then consider when $\alpha$ has exactly $n$ nonzero coefficients, say $c_1,c_2,\ldots,c_n$. Let $f(x) \in F[x]$ be the unique interpolant polynomial of degree $n$ such that $f(0) = f(c_1) =0$ and $f(c_i) = c_i$ for $1< i\le n$. This polynomial, of course, exists as it is the unique solution to an $F$-linear system involving the Vandemonde matrix associated to the coefficients $c_1, \ldots, c_n$. Then $f[\alpha]\in \S$, which is an element with $n-1$ nonzero coefficients. Then by the hypothesis $\overline{K(f[\alpha],c)}\in \S$ for all $c\in F\setminus \{0\}$. As each $c_i\ (i\neq 1)$ was mapped to itself by $f$, these coefficient complexes $\overline{K(f[\alpha],c_i)}$ must be the same complexes as they were for $\alpha$. Thus for all $c\in F\setminus\{0,c_1\}$ we have that $\overline{K(\alpha, c)} = \overline{K(f[\alpha],c)}\in \S$. Then $\overline{K(\alpha, c_1)} = \dfrac{1}{c_1}\left(\alpha - \sum_{i=2}^n c_i\overline{K(\alpha, c_i)}\right)\in \S$. Hence, for all $c\in F\setminus\{0\}$ we have that $\overline{K(\alpha,c)}\in \S$.
\end{proof} 

\begin{Thm}\label{thm:Smod} Let $G$ be a group and let $\S$ be a subspace of $\FG$. Then $\S$ is a Schur module if and only if $\S$ is closed under $\circ$ and for all $g\in G$ there exists some $\alpha\in \S$ such that $g\in \supp(\alpha)$.
\end{Thm}
\begin{proof}
Sufficiency is immediate from Proposition \ref{lem:had} and the existence of the partition $\D(\S)$. 

For necessity, let $\D$ be the set of supports of all primitive quantities of $\S$. We claim that $\D$ is a partition of $G$. Let $g\in G$. By assumption, there exists some $\alpha\in \S$ such that $g\in \supp(\alpha)$. By Lemma \ref{lem:Smod}, we may assume that $\alpha$ is simple. If $\alpha$ is imprimitive, then there exists some simple quantity $\beta$ such that $\alpha\circ \beta \neq 0$. If $g\in \supp(\beta)$, then $g\in \supp(\alpha\circ\beta)$. If $g\notin \supp(\beta)$, then $g\in \supp(\alpha-\alpha\circ\beta)$. In either case, $g$ is contained in a simple quantity whose support is strictly smaller than the support of $\alpha$. Repeating this process recursively, we see that exists some primitive quantity containing $g$ in its support. Suppose next that $\alpha, \beta\in \D$ such that $g\in \supp(\alpha)\cap \supp(\beta)$. But $\supp(\alpha)\cap\supp(\beta) = \supp(\alpha\circ\beta)$. If $\alpha\neq \beta$, then $\supp(\alpha\circ\beta) = \emptyset$, a contradiction. Therefore, $\D$ is a partition of $G$. As $\D$ consists of all the supports of the primitive quantities of $\S$, $\S = \Span_F\{\overline{C} \mid C\in \D\}$.
\end{proof}

When $G$ is finite, the element $\overline{G}$ serves as the supporting element for all $g \in G$. This allows Schur modules over finite groups to be viewed as $\Omega$-algebras in the sense of Universal Algebra, which allows for many universal algebraic results such as the intersection of two Schur modules is a Schur modules. For Schur modules over general groups, the absence of $\overline{G}$ can make these arguments fail, as illustrated in the following example. In particular, Schur modules of arbitrary groups do not form $\Omega$-algebra.

\begin{exmp} Consider the two Schur modules $\S$ and $\T$ over $\Z$ associated to the partitions
$\D(\S)=\{\{z^{2k},z^{2k+1}\}\mid k\in \mathbb{Z}\}$ and $\D(\T)=\{\{z^{2k},z^{2k-1}\}\mid k\in \mathbb{Z}\}$. Then $\S\cap \T = 0$, which is not a Schur module over $\Z$. Similarly, the common coarsening of two partitions is $\D(\S)\wedge\D(\T) = \{\Z\}$, which is not a partition of finite support.
\end{exmp}

Now that we have seen Schur modules over possibly infinite groups, we are ready to begin looking at Schur rings. This definition is do to Wielandt \cite{Wielandt49} and has been properly extending so that it can be for a finite or infinite group.

\begin{Def}[Schur Ring] \label{def:sring}
A \emph{Schur ring}, $\S$, over a group $G$ is a Schur module over $\FG$ such that
\begin{enumerate}[(i)]
\item $\{1\} \in \D(\S)$
\item if $C \in \D(\S)$, then $C^* \in \D(\S)$
\item\label{item:sring} for all $C, D \in \D(\S)$,
$\overline{C}\cdot\overline{D} = \sum_{E\in \D(\S)} \lambda_{CDE}\overline{E}$,
where all but finitely many $\lambda_{CDE}=0$.
\end{enumerate}
\end{Def}

The group ring $\FG$ itself forms a Schur ring over $G$ for any group, whose partition corresponds to singletons of group elements. For a finite group $G$, the partition $\{1, G\setminus 1\}$ induces a Schur ring known as the \emph{trivial Schur ring}. When $G$ is infinite, this partition is not of finite support. Thus, the trivial partition induces a Schur ring if and only if the group is finite. When $G=H\times K$ and $\S$ and $\T$ are Schur rings over $H$ and $K$, respectively, then $\S\otimes_F \T$ is a Schur ring over $G$, called the \emph{dot} (or \emph{tensor}) \emph{product} of $\S$ and $\T$. When $\H \le \Aut(G)$ is a finite automorphism subgroup, the set of elements of $\FG$ fixed by $\H$  is a Schur ring over $G$, denoted $\FG^{\H}$ and called the \emph{orbit Schur ring} associated to $\H$. When $G$ is abelian the partition $\{\{g, g^{-1}\} \mid g\in G\}$ induces a Schur ring, denoted $\FG^{\pm}$ and called the \emph{symmetric Schur ring}. This Schur ring corresponds to the orbit Schur ring associated to $\{\pm 1\}\le \Aut(G)$.
Likewise, the wedge product of Leung and Man \cite{LeungI} also applies to arbitrary groups so long as the subgroup is finite.

By taking supports, note that \eqref{item:sring} implies that $CD=\bigcup_{E\in \E} E$ for some finite subset $\E \subseteq \D(\S)$, that is, a product of primitive sets is likewise a union of primitive sets.

In light of Definition \ref{def:sring} and Theorem \ref{thm:Smod}, we extend Muzychuk's equivalence to arbitrary groups.

\begin{Cor}\label{cor:muzy} Let $G$ be a group and let $\S$ be a subring of $\FG$. Then $\S$ is a Schur ring if and only if $\S$ is closed under $\circ$ and $*$, $1\in \S$, and for all $g\in G$ there exists some $\alpha\in \S$ such that $g\in \supp(\alpha)$.
\end{Cor}

Like Schur modules, many of the properties of Schur rings remain true when $G$ is allowed to be infinite. Some of these results are included below, where new proofs are included only when the original proofs are insufficient to include infinite groups. 

\begin{Prop}[Wieldant, Proposition 23.5]\label{lem:stab}
Let $\S$ be a Schur ring over group G. Let $\alpha\in \S$ and $\stab(\alpha)= \{g\in G \mid \alpha g= \alpha\}$. then $\stab(\alpha)$ is an $\S$-subgroup of G.
\end{Prop}

\begin{Prop}[Wielandt, Proposition 23.6]\label{prop:sgroup}
Let $\S$ be a Schur ring over $G$. Let $\alpha\in \S$, and let $H=\langle \supp(\alpha) \rangle$. Then $H$ is an $\S$-subgroup of $\S$.
\end{Prop}

\begin{proof}
Let $L=\supp(\alpha)$ and $K=L\cup L^{*}$. Then $H= \langle L\rangle = \langle L, L^*\rangle = \langle K \rangle$. This implies that \[H=\bigcup_{n=0}^{\infty} K^{n}.\] Note that since $K=L\cup L^{*}$, where both $L$ and $L^*$ are unions of primitive sets, $K$ must also be a union of primitive sets of $\S$. As products of subsets distribute over unions and as products of primitive sets are unions of primitive sets, it follows by induction that $K^n$ is a union of primitive $\S$-sets for all $n$, that is, $K^n$ is an $\S$-set. As $H$ is a union of $\S$-sets, we conclude that $H$ is an $\S$-subgroup.
\end{proof}

In general, the intersection of two Schur modules (rings) over an infinite group need not be a Schur module (ring). The next proposition presents an important case when intersection provides a Schur module (ring) over a subgroup. 

\begin{Prop}\label{prop:sinter}
Let $\S$ be a Schur ring over a group $G$. Let $H$ be an $\S$-subgroup. Then $\S_H=\S\cap F[H]$ is a Schur ring over $H$.
\end{Prop}

\begin{proof}
As $\S$ is a Schur ring over $G$ and $\FH$ is a Schur ring over $H$, both rings are closed under $\circ$ and $*$ and $1\in \S_H$ by Corollary \ref{cor:muzy}. Thus, $\S_H$ is also a subring of $\FH$ closed under these operations. As $H$ is an $\S$-subgroup, we have that $H=\bigcup_{E\in \E} E$ where $\E\subseteq \D(\S)$. In particular, $\E$ is a partition of $H$, so for all $h\in H$ there exists some $E\in \E$ such that $h\in E = \supp(\overline{E})$. Thus, $\S_H$ is a Schur ring over $H$ by Corollary \ref{cor:muzy}.
\end{proof}


\begin{Prop}\label{prop:intsgroup} Let $\S$ be a Schur ring over a group $G$ and let $H, K\le G$ be $\S$-subgroups. Then $H\cap K$ is an $\S$-subgroup and $\S_{H\cap K} = \S_H\cap \S_K$.
\end{Prop}
\begin{proof}
Let $x\in H\cap K$. Then there exists some class $C\in \D(\S)$ such that $x\in C$. As $H$ is an $\S$-subgroup, it happens that $C\subseteq H$. Similarly, $C\subseteq K$. Thus, $C\subseteq H\cap K$. Then $H\cap K$ is a union of primitive sets of $\S$, which shows that $H\cap K$ is an $\S$-subgroup. The equality is immediate from the definition of $\S_{H\cap K}$ and properties of intersection.
\end{proof}

For every $\alpha = \sum _{g\in{G}} \alpha_g g \in{\S}$ where $\S$ is an S-ring over a group G. For $\alpha\in \FG$, we define the \emph{freshman exponent} of $\alpha$ by $n\in \mathbb{Z}$ as $\alpha^{(n)} = \sum _{g\in{G}} \alpha_g g^n$ and was first studied by Wielandt \cite{Wielandt64} in the context of Schur rings. The following identities are clear:
\begin{multicols}{2}
\begin{enumerate}[i)]
\item $(a\alpha+b\beta)^{(m)}=a\alpha^{(m)}+b\beta^{(m)}$
\item $\alpha^{(mn)}=(\alpha^{(m)})^{(n)}$
\item $\alpha^{( -1)}=\alpha^*$
\item $\overline{C}^{(a)}=\overline{C^{(a)}}$
\end{enumerate}
\end{multicols} 
\noindent for all $\alpha, \beta\in \FG$, $a,b\in F$, $m, n\in \mathbb{Z}$, and $C\subseteq G$.

\begin{Thm}[Weilandt, Theorem 23.9]\label{thm:fresh} 
If $\S$ is a Schur ring over $\FG$ where $G$ is a torsion-free abelian group, then for all $\alpha \in \S$ and $m\in \mathbb{Z}$, we have that $\alpha^{(m)}\in \S$, that is, Schur rings over $G$ are closed under freshman exponents. \end{Thm}

The above theorem was proven by Wielandt \cite{Wielandt64} for finite groups. The proof of Theorem \ref{thm:fresh} is essentially the same as Wielandt's in the case that the exponent $m$ is relatively prime to the order of the group. In reality what is needed for Wieldant's argument to hold is that $m$ is relatively prime to all the orders of elements in $\supp{\alpha}$ which is the case for torsion-free groups. 

\section{Classifying Schur Rings Over the Integers}\label{sec:cyclic}

Now that we see that Schur rings over $F[\Z]$ are closed under freshman exponentiation, we can move onto proving that there are only two Schur rings over $F[\Z]$, the group ring and the symmetric Schur ring.



\begin{Lem}\label{lem:key}
Let $\S$ be a Schur ring over $\Z=\langle z\rangle$. Let $C\in \D(\S)$. If $z \in C$ then $C=\{z\}$ or $C=\{z, z^{ -1}\}$. \end{Lem}
\begin{proof}
Let $w\in C$. Then $w=z^{a}$ for some $a\in\mathbb{Z}$. Then $w\in C\cap C^{(a)}$. As $\overline{C}\in \S$, $\overline{C}^{(a)}\in \S$ by Theorem \ref{thm:fresh}. Now, $\overline{C\cap C^{(a)}} = \overline{C}\circ \overline{C^{(a)}} = \overline{C}\circ \overline{C}^{(a)}\in \S$. Since $\overline{C}$ is a primitive quantity,  $\overline{C}\circ\overline{C}^{(a)}=r\overline{C}$ for some $r\in F$, which necessarily must be $r=1$ or $r=0$ as $\overline{C^{(a)}}$ is a simple quantity. As $C\cap C^{(a)}\neq \emptyset$, it must be the former, that is, $\overline{C}\circ\overline{C}^{(a)}=\overline{C}$. This implies that $C\subseteq C^{(a)}$. But $|C| = |C^{(a)}| < \infty$. We conclude that $C=C^{(a)}$. This means $z\in C^{(a)}$, that is, $z=z^{ax}$ for some $x\in \mathbb{Z}$. Thus, $1=ax$, which implies that $a=1$ and $x=1$, or $a=-1$ and $x=-1$. Hence, $w=z$, or $w=z^{-1}$, which proves the result.
\end{proof}

\begin{Thm}\label{thm:int} The only Schur rings over the infinite cyclic group are the group ring and the symmetric ring. \end{Thm}

\begin{proof}
Let $\S$ be a Schur ring over $\Z$. Let $C \in \D(\S)$ be the primitive set containing $z$. Then by, Lemma \ref{lem:key}, we consider two cases: $C = \{z\}$ or $C=\{z,z^{-1}\}$. If $C=\{z\}$, then $z^m\in \S$ for all $m\in \mathbb{Z}$. Thus, $F[\Z]\subseteq \S$, forcing $\S=F[\Z]$, the group ring. 

If $C=\{z,z^{-1}\}$, then $z+z^{-1}\in \S$. Then $=z^a+z^{-a}=(z+z^{ -1})^{(a)}\in \S,\ \forall a\in \mathbb{Z}$ by Theorem \ref{thm:fresh}. That is, $\{z^a,z^{-a}\}$ is an $\S$-set. Now, by way of contradiction, assume that $\exists m\in{\mathbb{Z}}$ such that $\{z^{m},z^{-m}\}$ is imprimitive for $m\neq 0$. Then $\{z^{m}\}$ and $\{z^{-m}\}$ are primitive sets of $\S$, that is, $z^{m}, z^{-m}\in \S$. Then $z^{m}(z+z^{-1}) = z^{m+1}+z^{m-1}\in\S$. We also know that $z^{m+1}+z^{-(m+1)}\in \S$, from freshman exponentiation. Then using the Hadamard product, $(z^{m+1}+z^{m-1}) \circ (z^{m+1}+z^{-m+1)})=z^{m+1}\in{\S}$. Hence, $z^{m+1}z^{-m} = z\in\S$, a contradiction. Hence, $\{z^{a},z^{-a}\}$ is a primitive set of $\S$, for all $a\in{\mathbb{Z}}$. In this case $\S$ is the symmetric ring $F[\Z]^\pm$. 
\end{proof}

Note that because $\Z$ is infinite, the trivial Schur ring is not possible over $\Z$. Likewise, no dot product is constructible over $\Z$ since it has no nontrivial direct product factorization. Finally, as $\Z$ no nontrivial finite subgroups, no wedge product is constructible over $\Z$. Thus, of the four types of Schur rings available to finite cyclic groups by the Leung-Man Classification, orbit Schur rings are the only possible type available for $\Z$. Furthermore, as $\Aut(\Z) = \{\pm 1\}$, $\FG$ and $\FG^\pm$ are the two orbit Schur rings over $\Z$. Thus, we see that Theorem \ref{thm:int} is an extension of the Leung-Man Classification of Schur rings to all cyclic groups, both finite and infinite.

\section{Classifying Schur Rings Over the Torsion-Free Locally Cyclic Groups}\label{sec:rational}
Now that we have a classification of Schur rings over the infinite cyclic group, we are ready to move onto classifying Schur rings over all torsion-free locally cyclic groups. It is well known that a group is torsion-free locally cyclic if and only if it is isomorphic to a subgroup of the additive group of rational numbers. To do this we will first look at some results concerning $\S$-sets and $\S$-subgroups.

\begin{Thm}\label{Thm:TFLC}
Let $\S$ be a Schur ring over $G$, where $G$ is a torsion-free locally cyclic group. Then $\S$ is either the group ring $\FG$ or the symmetric Schur ring $\FG^\pm$.
\end{Thm}

\begin{proof}
Let $\S$ be an S-ring over $G$, where $G$ is a torsion-free locally cyclic group. Let $C\in \D(\S)$ excluding $\{1\}$. As $C$ is necessarily finite, the subgroup $H=\langle C\rangle \cong \Z$ and is an $\S$-subgroup by Proposition \ref{prop:sgroup}. Then $\S_H = \FH$ or $\S_H=\FH^\pm$ by Theorem \ref{thm:int}. As $\overline{C}\in \S_H$, we see that $C=\{z^a\}$ or $C=\{z^a,z^{-a}\}$ for some $a\in \mathbb{Z}$.

Let $C,D\in \D(\S)$ excluding $\{1\}$. Consider $H=\langle C \rangle,\ K=\langle D \rangle$, which are necessarily cyclic. Like above, $\S_H=\FH$ or $\S_H=\FH^\pm$, and $\S_K =\FK$ or $\S_K= \FK^\pm$. This presents us with three cases:
\begin{enumerate}[(i)]
\item\label{item1:TFLC} $\S_H = \FH$ and $\S_K=\FK$;
\item\label{item2:TFLC} $\S_H = \FH$ and $\S_K=\FK^\pm$;\footnote{The case $\S_H = \FH^\pm$ and $\S_K=\FK$ is similar to \eqref{item2:TFLC} and is omitted.}
\item\label{item3:TFLC} $\S_H = \FH^\pm$ and $\S_K=\FK^\pm$.
\end{enumerate} 

Note that $H\cap K\cong \Z$ and $\S_{H\cap K} = \S_H\cap \S_K$, by Proposition \ref{prop:intsgroup}. As $\S_{H\cap K}$ is an $\S$-subgroup, $\S_{H\cap K}=F[H\cap K]$ or $\S_{H\cap K} = F[H\cap K]^\pm$.

In case \eqref{item1:TFLC}, $\S_{H\cap K} = \FH\cap\FK = F[H\cap K]$. Likewise, in case \eqref{item3:TFLC}, $\S_{H\cap K} = \FH^\pm\cap\FK^\pm = F[H\cap K]^\pm$, as both sets consist of all elements of $F[H\cap K]$ fixed by $\H=\{\pm 1\}$.

We next will show that case \eqref{item2:TFLC} is impossible. Let $x\in H\cap K$. As $\S_H=\FH$ and $x\in H$, it must be that $\{x\}\in \D(\S)$. As $\S_K=\FK^{\pm}$ and $x\in K$, it must be that $\{x,x^{-1}\}\in \D(\S)$, which is a contradiction. This means that $\S_{H}$ and $\S_{K}$ cannot be different types of Schur rings, that is, they are both group rings or both symmetric rings over their respective subgroups. In summary, every class in $\D(\S)$ is either a singleton or a symmetric pair and either all classes are singletons or all classes are pairs. Hence $\S$ is either the group ring over $G$ or the symmetric ring over $G$.
\end{proof}

\begin{Cor} Let $\mathcal{Q}$ be the group of rational numbers. Then the only Schur rings over $\mathcal{Q}$ are the group ring $F[\mathcal{Q}]$ and the symmetric Schur ring $F[\mathcal{Q}]^\pm$.
\end{Cor}

\bibliographystyle{plain}
\bibliography{Srings}
\end{document}